\documentclass[english]{amsart}

\usepackage{hyperref}
\usepackage{babel}
\usepackage{amsmath,amssymb,amsthm}
\usepackage{mathtools}
\usepackage[latin1]{inputenc}
\usepackage[T1]{fontenc}
\usepackage{graphicx}
\usepackage{tikz}

\usepackage[all]{xy}
\SelectTips{cm}{10}

\DeclareMathOperator{\Wh}{Wh}

\DeclareMathOperator{\Gl}{Gl}

\DeclareMathOperator{\id}{Id}

\newcommand{\R}{\mathbb{R}}

\newcommand{\Z}{\mathbb{Z}}
\newcommand{\bR}{\mathbb{R}}

\newcommand{\bC}{\mathbb{C}}

\newcommand{\bZ}{\mathbb{Z}}

\newcommand{\cA}{\mathcal{A}}

\newcommand{\cM}{\mathcal{M}}

\newcommand{\ro}{\mathrm{o}}

\newcommand{\Spin}{\mathrm{Spin}}

\renewcommand{\epsilon}{\varepsilon}

\newcommand{\cX}{\widetilde{X}}
\newcommand{\cY}{\widetilde{Y}}
\newcommand{\cw}{\textrm{CW}}

\DeclarePairedDelimiter\absv{\lvert}{\rvert}
\DeclarePairedDelimiter\norm{\lVert}{\rVert}

\newtheorem{theorem}{Theorem}
\newtheorem{proposition}{Proposition}[section]
\newtheorem{lemma}[proposition]{Lemma}
\newtheorem{corollary}[proposition]{Corollary}

\theoremstyle{remark}
\newtheorem{remark}[proposition]{Remark}

\numberwithin{equation}{section}

\theoremstyle{definition}
\newtheorem{definition}[proposition]{Definition}

\def\co{\colon\thinspace}

\newcommand{\superscript}[1]{\ensuremath{^{\textrm{#1}}} }

\renewcommand{\th}[0]{\superscript{th}}

\begin{document}

\title[Simple homotopy equivalence of nearby Lagrangians]{Simple homotopy equivalence of nearby Lagrangians}

\author[M. Abouzaid]{Mohammed Abouzaid}
\author[T. Kragh]{Thomas Kragh}
\thanks{MA was supported by NSF grant DMS-1308179, and by Simons Foundation grant 385571. The authors were hosted by the Mittag-Leffler institute during its ``Symplectic Geometry and Topology'' program.}

\begin{abstract} Given a closed exact Lagrangian in the cotangent bundle of a closed smooth manifold, we prove that the projection to the base is a simple homotopy equivalence. 
\end{abstract}

\date{\today}


\maketitle
\setcounter{tocdepth}{1}
\section{Introduction}
Let $N$ be a closed smooth manifold. In accordance with conjectures of Arnol'd and Eliashberg, the symplectic topology of the cotangent bundle $T^*N$ is expected to be deeply tied to the smooth topology of $N$.  One way to study the symplectic topology of $T^*N$ is to consider its Lagrangian submanifolds; symplectic topologists have traditionally focused on those Lagrangians which are exact with respect to the canonical Liouville form on the cotangent bundle.  The main result of this paper is the following:
\begin{theorem}
  \label{thm:1}
  Let $L \subset T^*N$ be a closed exact Lagrangian. The canonical map $L\subset T^*N \to N$ is a simple homotopy equivalence.
\end{theorem}

We shall presently explain our strategy for proving this result, after highlighting the following consequence: recall that a lens space is a quotient of $S^{2n-1}  \subset \bC^{2n}$ by a finite cyclic subgroup of $U(n)$.  It was classically known following Franz and Reidemeister that the simple homotopy type of lens spaces determines their diffeomorphism type (a modern account is provided by Milnor in \cite[Section 9]{MR0196736}). From Theorem \ref{thm:1}, we therefore conclude an answer to a long-standing open question:
\begin{corollary}
A pair of lens spaces are diffeomorphic if and only if their cotangent bundles are symplectomorphic. \qed
\end{corollary}

The proof of Theorem \ref{thm:1} can be broken into three different steps: first, we introduce a notion of \emph{Whitehead torsion} for a Floer theoretic equivalence. Here, a Floer theoretic equivalence should be interpreted as an equivalence in the Fukaya category of a symplectic manifold, though our usage of this machinery is quite limited, and we give a self-contained discussion of this notion in terms of the existence of Floer cocycles whose products satisfy a certain non-degeneracy condition. At a technical level, the study of torsion invariants for equivalences is a minor generalisation of the work of Sullivan \cite{MR1935550} and Hutching and Lee \cite{MR1716272} whose work considered the torsion of acyclic Floer complexes. However, we find it convenient to provide our own account for the construction of this invariant, as the specific context we are studying allows for greatly simplified proofs. This is the content of Section \ref{sec:whit-tors-fukaya}.

The next step is to appeal to the results of \cite{Nadler,FSS,Abou1,MySympfib} showing that every closed Lagrangian in a cotangent bundle gives rise to an object of the Fukaya category which is Floer-theoretically equivalent to the $0$-section. In particular, it makes sense to consider the torsion of this Floer-theoretic equivalence, and we prove that this vanishes in Section \ref{sec:spectr-sequ}. This section is at the heart of the paper, and the key idea that it uses is a large-scale Hamiltonian deformation (see also \cite{HomotopySerre}) which allows us to reduce the computation of torsion to a computation at the level of certain associated graded groups which admit a particularly simple description.

Finally, we show in Section~\ref{sec:from-cellular-floer} that the Floer-theoretic torsion which we introduce agrees, in the setting of cotangent bundles, with the classical torsion. Here, the main ingredients is the compatibility of Floer and Morse theoretic constructions in the exact setting, reminiscent of the constructions of Puinikin-Schwarz-Salamon relating Floer and Quantum cohomology.

From the above ingredients, the proof of Theorem \ref{thm:1} is immediate. We make it explicit for the record:
\begin{proof}
Proposition \ref{prop:CWtoFloer:1} shows that the Whitehead torsion of the map $L \to N$ agrees with the  torsion of the Floer theoretic equivalence between $L$ and $N$. Proposition \ref{prop:floer_torsion_vanishes} implies that this torsion vanishes.
\end{proof}

\subsection*{Acknowledgments:} The first author would like to thank Shmuel Weinberger and Dennis Sullivan for explanations about the role of simple homotopy theory in the classification of manifolds. The second author would like to thank John Rognes for the same reason.

\section{Whitehead Torsion} \label{sec:whitehead-torsion}

In this section we give a crash course in Whitehead torsion and Whitehead groups (more details can be found in \cite{MR0035437}). We also prove a lemma about filtrations and Whitehead torsion which will be needed at several points later. Note that everything here is formulated in terms of $\Z$-graded chain complexes, but it works equally well for $\Z/2$ graded complexes (with one exception, which we will point out).

Let $R=\Z[G]$ be the group ring of a discrete group $G$, and let $C_* = (R[\alpha_i],d)$ be a finitely generated acyclic chain complex of left $R$ modules freely generated by some $\alpha_i$ with degree $\absv{\alpha_i}$  - we assume that the differential always increases the degree. The Whitehead torsion of the pair $(C_*, \{ \alpha_i \})$ is defined as the equivalence class under the relations generated by
\begin{itemize}
\item A \emph{very simple expansion} given by direct sum with the chain complex 
$$\cdots \to 0 \to R \xrightarrow{\id} R \to 0 \to \cdots $$
with the obvious choice of generators.
\item A \emph{handle slide} given by replacing a basis element by itself plus a linear combination of the other basis elements in the same degree.
\item A \emph{simple base change} given by replacing any $\alpha_i$ with $(\pm g)\alpha_i$ for $g\in G$.
\end{itemize}
These equivalence classes form a group $\Wh(G)$ (the Whitehead group of $G$) under direct sum. We speak of the Whitehead torsion of a complex when the choice of basis is self-evident (as it is for Floer, Morse, and simplicial complexes).

Using the chosen basis (and choosing an auxiliary ordering) we can identify each differential with a matrix with coefficients in $\Z[G]$. Whitehead proved in \cite{MR0035437}  that using these moves any chain complex is equivalent to one with support in \emph{any} two adjacent degrees, and a single non-zero differential, which then has to be an isomorphism in $\Gl_n(\Z[G])$. Note, that this does not hold for $\Z/2$ graded complexes; and therefore the analogue of the Whitehead group in that case can be larger.  The inverse in the Whitehead group is given by the chain complex with the inverse differential. Indeed, one can show using simple base changes and handle slides that the differential
\begin{align*}
  \begin{pmatrix}
    A & 0 \\
    0 & A^{-1}
  \end{pmatrix} \sim
  \begin{pmatrix}
    AA^{-1} & 0 \\
    0 & I
  \end{pmatrix} = I.
\end{align*}
Here $\sim$ should be interpreted as equality in the Whitehead group between the two-term chain complexes obtained from the two matrices.

For any chain homotopy equivalence $f:(C,d_C) \to (D,d_D)$ of left $\Z[G]$-module chain complexes (with bases) we can define the mapping cone as
\begin{align*}
  (C[1] \oplus D, d)=
  \begin{pmatrix}
    d_C[1] & f \\
    0 & d_D
  \end{pmatrix}
\end{align*}
(with the direct sum basis). Here $C[1]$ means the chain complex shifted up by  $1$ in degree and changing the sign on the differential. This is an acyclic chain complex and so defines an element $\tau(f)$ in $\Wh(G)$. An important property of this that we will use repeatedly is that
\begin{align} \label{eq:5}
  \tau(f \circ g) = \tau(f)+\tau(g)
\end{align}
for the composition of chain homotopy equivalences.

Now let $X$ and $Y$ be finite CW complexes and $f\co X\to Y$ a cellular homotopy equivalence with $G=\pi_1(Y) \cong \pi_1(X)$. By considering the universal covering spaces $\cX \to X$ and $\cY \to Y$ with their natural CW structure we get (after a choice of base-points compatible with $f$) an induced chain homotopy equivalence
\begin{align*}
  \widetilde{f}_* \co C_*(\cX,\Z) \to C_*(\cY,\Z).
\end{align*}
Here our convention is that a $n$-dimensional cell contributes a generator of degree $-n$, so that the differential raises degree as with the rest of the chain complexes we shall study. Using the choice of base-points and choosing a lift of each cell in $X$ and $Y$ this lifts to a chain homotopy equivalence of $\Z[G]$-modules complexes (complete with choices of basis as above)
\begin{align*}
  \widetilde{f}_* \co C_*^{\cw}(X,\Z[G]) \to C_*^{\cw}(Y,\Z[G]).
\end{align*}
using the mapping cone construction as above. This defines the Whitehead torsion $\tau(f)=\tau(\widetilde{f}_*)\in \Wh(G)$ of $f$. This is well-defined since different choices of lifts of cells corresponds to simple base changes, and so does the choice of base-point (but it affects many generators simultaneously). The Whitehead torsion of $f$ is preserved under the notion of a simple expansion of $X$ (or $Y$) given by
\begin{itemize}
\item A \emph{simple expansion} is given by attaching the two cells (with positive $x_{n+1}$) in an upper half sphere $D^{n+1}\cap \{x_{n+1}\geq 0\}$ by any cellular map $D^n \to X$.
\end{itemize}
The inverse of this is called s \emph{simple contraction}. A simple expansion can be divided into to two moves (after choosing a base-point):
\begin{itemize}
\item A \emph{very simple expansion} given by attaching the two cells by the trivial map $D^n \to X$. This is the same as wedging with $D^{n+1}$ (having two cells and a base-point).
\item A \emph{handle slide} given by changing any attaching map of a cell to a homotopic cellular attaching map. By the homotopy extension property this extends to a homotopy equivalence of $X$ to the new representative of $X$ (well-defined up to handle slides of the higher dimensional cells).
\end{itemize}
One can prove that these two generates the same equivalence classes of finite CW complexes as the one move above - called the simple homotopy type. However we are more interested in the relative version for the map $f$, which asks whether the map is a simple homotopy equivalence, which is stronger than asking if there is an abstract simple homotopy equivalence between $X$ and $Y$. This requires the moves to be done on the pair $(M_f,X)$ where 
\begin{align*}
  M_f=X\times I \cup Y / (x,1)\sim f(x)
\end{align*}
is the mapping cylinder and $X=X\times \{0\}$ is the sub CW complex. Applying the moves on a pair $(A,B)$ means applying them to $A$ without changing the sub-space $B$, and $f$ is called a simple homotopy equivalence if $(M_f,X)\sim (X,X)$. Whitehead proved that this is equivalent to $\tau(f)=0$ in the Whitehead group. The reason that one does not pass to the quotient space $M_f/X$ is that it kills the fundamental group and thus there are no longer any obstructions for a homotopy equivalence to be a simple homotopy equivalence. However, one way to think of this as a quotient is to instead consider the universal covers quotient: $\widetilde{M_f}/\widetilde{X}$ which have a free action of $\pi_1$ (except on the base-point) and the moves are really $\pi_1$ equivariant moves (adding a $\pi_1$ worth of cells at a time), and then the question is whether or not we can perform such moves to get to the trivial space (consisting of just the base-point).


Whitehead torsion do not generally behave well when it comes to filtrations and spectral sequences, but we will need the following algebraic lemma, which is a very special case of using a spectral sequence to compute Whitehead torsion.

\begin{lemma}
  \label{lem:Whitehead:1}
  Let $F^pC$ be a filtered acyclic chain complex (of free based $\Z[G]$ left modules) such that each filtered quotient $F^pC/F^{p-1}C$ is free (using a subset of the same generators), acyclic, and has trivial Whitehead torsion. Then $C$ has trivial Whitehead torsion.
\end{lemma}

Note that the condition on the filtration to respect the basis in this way really means that it is a filtration of the basis elements.

\begin{proof}
  Since we have finitely many generators the filtration is finite (and we may assume it starts with $F^0C=0$). Let $k$ be minimal such that $F^{k+1}C=C$. We will prove that $C$ is equivalent in the Whitehead group to $F^kC$ and hence the lemma follows by induction.

  Consider the entire chain complex as one module (ignoring degree) and order the basis such that it starts with the basis elements in the top filtration degree and descends trough the filtration. This means that the differential for the entire complex can be written in the following upper triangular block form:
  \begin{align*}
    \begin{pmatrix}
      d_{k+1} & ? & \cdots & ? \\
      0 & d_{k} &  \cdots & ? \\
      \vdots & \ddots & \ddots & \vdots \\
      0 & \cdots & 0 & d_1 \\
    \end{pmatrix}
  \end{align*}
  where $d_p$ is the differential on $F^pC/F^{p-1}C$. Now since the top quotient $F^{k+1}C/F^kC$ has trivial Whitehead torsion there is a sequence of moves reducing it to the $0$ complex. Using the same sequence of moves but expanded by extra handle slides just before canceling a pair of generators we see that the complex is equivalent in the Whitehead group to $F^{k}C$. Indeed, right before we (in the original sequence) do a cancellation the differential on the top-left block has a 1 which has only zeros above, below, to its left and to its right (inside the block). However, to be able to extend this cancellation to all of $C$ we need to clear the possibles non-zero elements further to the right of it, but this is easily done by adding an extra handle slide. Indeed, if this situation represents $da=b+$(lower filtration terms), where $a$ and $b$ are generators in the top filtration - then we simply replace $b$ by $b+$(these lower filtration terms) and consider this new generator as still being a generator in top filtration. This way the complex is still filtered in the way specified by the lemma.
\end{proof}

The following lemma is not immediately clear from the definition, but will be needed to prove perturbation invariance in Floer and Morse theoretic contexts. However, the lemma should be rather intuitive given the description above of the moves as $G$ equivariant moves. Indeed, a $\Z[G]$ homotopy is like a $G$ invariant homotopy, and hence can be realized by a sequence of $G$ invariant simple moves.

\begin{lemma}
  \label{lem:Whitehead:2}
  If $f,g : C_* \to D_*$ are chain homotopic maps (as left $\Z[G]$-modules) then their Whitehead torsions agree.
\end{lemma}

\begin{proof}
  Let $\Phi : C[1] \to D$ be a left $\Z[G]$-module chain homotopy from $f$ to $g$, i.e. we have $d\Phi+\Phi d = f-g$. Then there are sequences of moves showing the equivalence of the chain complexes
  \begin{align*}
    C[1]\oplus D \sim C[2]\oplus C[1] \oplus C[1] \oplus D \sim C[1]\oplus D
  \end{align*}
  with differentials
  \begin{align*}
    \begin{pmatrix}
      d_C[1] & f \\
      0 & d_D
    \end{pmatrix} \sim 
    \begin{pmatrix}
      d_C[2] & I & I & \Phi \\
      0 & d_C[1] & 0 & f \\
      0 & 0 & d_D[1] & -g \\
      0 & 0 & 0 & d_D
    \end{pmatrix} \sim
    \begin{pmatrix}
      d_C[1] & -g \\
      0 & d_D
    \end{pmatrix}
          \sim
    \begin{pmatrix}
      d_C[1] & g \\
      0 & d_D
    \end{pmatrix}
  \end{align*}
The moves can be obtained as in the proof of Lemma  \ref{lem:Whitehead:1} by considering the two canonical two-stage filtrations of the middle complex
  \begin{align*}
    C[1] \oplus D \subset C[2]\oplus C[1] \oplus C[1] \oplus D;
  \end{align*}
the  filtrations differ by which factor of $C[1]$ we include (the presence of a $0$ entry above the diagonal ensures that these are both subcomplexes). In both cases, the first quotient $F^2/F^1$ has no Whitehead torsion (since it is equivalent to the mapping cone of the identity) so we can proceed to remove it as in the proof of the Lemma  \ref{lem:Whitehead:1}.
\end{proof}

\section{Whitehead torsion and the Fukaya category} \label{sec:whit-tors-fukaya} 
Whitehead groups have been defined and used before in the context of intersection Floer homology, the work by M. Sullivan in  \cite{MR1935550} is the most relevant for this paper (see also \cite{MR1716272} which studies a much harder problem). Our treatment will however be essentially self-contained, in part because we will discuss torsion for \emph{Floer theoretic equivalences,} but also because we will only consider such groups in the exact setting.
This means that the action filtration will allow us to use Lemma~\ref{lem:Whitehead:1} to give a much shorter proof of the invariance of Whitehead torsion. We shall use Seidel's book \cite{MR2441780} as reference for various points in Lagrangian Floer theory.

\subsection{Floer cohomology and the action filtration}
\label{sec:floer-diff-prod}

We begin by considering the following context: let  $(M, \lambda)$ be a Liouville domain with symplectic form $\omega = d \lambda$, equipped with a \emph{grading} $\widetilde{Gr}_{\Lambda} (M)$ and a \emph{background class} $b \in H^2(M, \Z_2)$. Here, the grading is a choice of cover of the Grasmannian bundle of Lagrangian subspace in $TM$, which restricts over each point in $M$ to a universal cover for the Grasmannian of linear Lagrangians (see \cite[Section 11h]{MR2441780}). To simplify the exposition, we shall assume throughout that $b$ is the second Stiefel-Whitney class of an orientable vector bundle $V$ on $M$. In the example of interest, $M$ is the cotangent bundle of a smooth manifold $N$, the cotangent fibres lift to a section of the preferred grading and $V$ is the pullback of $TN \oplus (\det N)^{\oplus 3}$, whose second Stiefel-Whitney class is the pullback of $w_2 (N)$.

Consider an exact Lagrangian $L$ which is a \emph{brane} relative to the background class $b$: this means that we fixed  a function $f \co L \to \bR$ such that $df = \lambda_{|L}$, and $L$ is equipped with a lift to $\widetilde{Gr}_{\Lambda} (M) $, as well as a $\Spin$ structure on the direct sum of the tangent space with the restriction of the vector bundle $V$. The first structure implies that the Lagrangian has vanishing Maslov class, and the second that the second Stiefel-Whitney class agrees with the restriction of $b$. In the case of a cotangent bundle, it is clear that the zero section is a brane relative the pullback of $w_2 (N)$.



If $(L_0,f_0)$ and $(L_1,f_1)$ are transverse branes, we can associate to each intersection point $ x \in L_0 \cap L_1$ a \emph{Maslov index} $\deg(x) \in \Z$ and a graded one dimensional real vector space, which we denote $\ro_x$ as in \cite[Section 12f]{MR2441780} (see also \cite[Appendix A]{Abou2}), so that, whenever $J$ is a generic $\omega$-tame almost complex structure on $M$, the following holds: for each pair $(x,y)$ intersection points such that $\deg(y) +1 = \deg(x)$, every element $u$ of the moduli space $\cM(x,y)$  of $J$-holomorphic strips induces a canonical isomorphism
\begin{equation} \label{eq:map_induced_strip}
  \kappa_u \co \ro_x \to \ro_y.
\end{equation}
More precisely, we define $\cM(x,y)$ to be the quotient by $\bR$-translation of the space of holomorphic maps from $\bR \times [0,1]$ to $M$  with boundary conditions $L_0$ along $ \bR  \times \{0\} $ and $L_1$ along $ \bR \times \{1\} $, with asymptotic conditions $x$ at $-\infty \times [0,1] $  and $x$ at $\infty \times [0,1] $.

The Lagrangian Floer cochains are then defined to be the direct sum
\begin{equation} \label{eq:Floer_cochains}
  CF^*(L_0,L_1;\Z) = \bigoplus_{x  \in L_0 \cap L_1} |\ro_x|,
\end{equation}
where $ |\ro_x|$ is the orientation line of $\ro_x$ (i.e. the free abelian group generated by the two orientation of this vector space, with the relation that their sum vanishes). The differential is the sum, over all $u \in  \cM(x,y)$,  of the maps on orientation lines induced by Equation \eqref{eq:map_induced_strip}. 

Since we are considering the exact setting, we can associate to each intersection point its \emph{action}
\begin{equation}
\cA(x) = f_1(x) - f_0(x) \in \bR.
\end{equation}
An easy consequence of Stokes's theorem and the fact that $\int u^* \omega$ is positive whenever $u$ is a holomorphic curve is that the moduli space $\cM(x,y)$ is empty whenever $\cA(x) \leq \cA(y)$ and $x \neq y$ (see \cite[Section 8g]{MR2441780}). At the level of Floer complexes, we obtain a filtration by subcomplexes
\begin{align*}
  F^{a}CF^*(L_0,L_1;\Z) = \bigoplus_{\cA(x) \geq a} |\ro_x|.
\end{align*}
The reader should keep in mind that, with the conventions we have adopted, the differential \emph{raises} action.

Whenever $L = L_0 = L_1$, we follow Biran-Cornea\cite{MR2546618}, Seidel \cite{MR2819674} and Sheridan \cite{MR3294958}, and pick a Riemannian metric together with a Morse-Smale function $h: L \to \bR$ with a unique minimum, and define
\begin{equation}
 CF^*(L, L; \bZ) = CM^*( L, h; \bZ)
\end{equation}
where the right hand side are Morse cochains. The minimum of $h$ defines a cocycle $e \in CM^0( L, h; \bZ)$ which, at the level of cohomology, corresponds to the unit in $[e] \in H^0(L, \bZ)$ under the isomorphism of  Morse cohomology with ordinary cohomology. We set the action filtration on $ CF^*(L, L; \bZ)$ so that all generators have $0$ action.

\subsection{The product on Floer cohomology}
\label{sec:prod-floer-cohom}


Given a triple of Lagrangian branes $(L_0, L_1, L_2)$, which are in transverse position, we have a co-chain map
\begin{equation} \label{eq:Floer_product}
  CF^*(L_1,L_2;\Z) \otimes  CF^*(L_0, L_1; \Z) \xrightarrow{\mu_{2}} CF^*(L_0,L_2;\Z)
\end{equation}
given by a count of rigid triangles.  More precisely, say that $(J_{01}, J_{12}, J_{02})$ are a triple of almost complex structures used to define the above Floer  complexes. Choosing a (generic) family $J$ of almost complex structures on the disc, which agrees at three boundary marked points (ordered clockwise) with the above triple, we obtain a  moduli space of $\cM(x_0, x_1, x_2)$ associates to  every triple of intersection points $x_0 \in L_0 \cap L_2$,  $x_1 \in L_0 \cap L_1$ and $ x_2 \in L_1 \cap L_2$ as follows: elements of $\cM(x_0, x_1, x_2) $ are holomorphic maps from the disc to $M$,  mapping the three marked points to $x_0$, $x_1$, and $x_2$, with boundary conditions along the boundary given by  the triple of Lagrangians $L_0$, $L_1$ and $L_2$ (see Figure~\ref{fig:db}).
\begin{figure}[ht]
  \centering
  \begin{tikzpicture}
    \draw[fill=lightgray] (-1.8,0) circle (0.8);
    \draw (-2.1,+1.0) node {$L_0$};
    \draw (-1.0,0) node [right] {$L_1$};
    \draw (-2.1,-1.0) node {$L_2$};
    \fill (-2.6,0) circle (2pt) node [left] {$x_0$};
    \fill (-1.29,0.61) circle (2pt) node [above right] {$x_1$};
    \fill (-1.29,-0.61) circle (2pt) node [below right] {$x_2$};
  \end{tikzpicture} 
  \caption{Disc with marked points and boundary conditions}  \label{fig:db}
\end{figure}
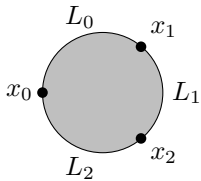
Whenever two of these Lagrangians agree, we instead count configurations consisting of holomorphic discs and (perturbed) gradient flow lines (see Figure~\ref{fig:dwg}).
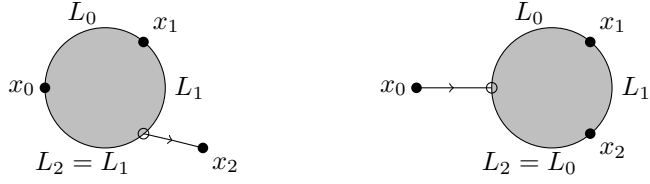
\begin{figure}[ht]
  \centering
  \begin{tikzpicture}
    \draw[fill=lightgray] (-1.8,0) circle (0.8);
    \draw (-2.1,+1.0) node {$L_0$};
    \draw (-1.0,0) node [right] {$L_1$};
    \draw (-2.1,-1.0) node {$L_2=L_1$};
    \fill (-2.6,0) circle (2pt) node [left] {$x_0$};
    \fill (-1.29,0.61) circle (2pt) node [above right] {$x_1$};
    \draw (-1.29,-0.61) circle (2pt);
    \draw[->] (-1.29,-0.61) -- (-0.9,-0.7);
    \draw (-0.9,-0.7) -- (-0.5,-0.8) node [below right] {$x_2$};
    \fill (-0.5,-0.8) circle (2pt);   
  \end{tikzpicture} {$\qquad\qquad$}
  \begin{tikzpicture}
    \draw[fill=lightgray] (-1.8,0) circle (0.8);
    \draw (-2.1,+1.0) node {$L_0$};
    \draw (-1.0,0) node [right] {$L_1$};
    \draw (-2.1,-1.0) node {$L_2=L_0$};
    \draw (-2.6,0) circle (2pt);
    \draw (-2.6,0) -- (-3.1,0);
    \draw[<-] (-3.1,0) -- (-3.6,0);
    \fill (-3.6,0) circle (2pt) node [left] {$x_0$};
    \fill (-1.29,0.61) circle (2pt) node [above right] {$x_1$};
    \fill (-1.29,-0.61) circle (2pt) node [below right] {$x_2$};
  \end{tikzpicture}   
  \caption{Discs with gradient trajectory attached} \label{fig:dwg}
\end{figure}
Applying this to the special case when $L_1 = L_2$, and two choices of almost complex structures $J_{01}$ and $J_{01}'$, we obtain a map
\begin{equation}
CM^*(L_1,L_1;\Z) \otimes  CF^*(L_0, L_1, J_{01}; \Z) \xrightarrow{} CF^*(L_0,L_1, J_{01}';\Z),
\end{equation}
where we temporarily introduce notation which incorporates the choice of almost complex structure.
If we input the class of the minimum in $ CM^*(L_1,L_1;\Z)  $, we obtain a map 
\begin{equation}
  CF^*(L_0, L_1, J_{01}; \Z) \xrightarrow{} CF^*(L_0,L_1, J_{01}';\Z)
\end{equation}
which is a \emph{continuation map} for almost complex structures.
\begin{lemma} \label{lem:continuation_chain_iso}
The continuation map for almost complex structures is a chain isomorphism.
\end{lemma}
\begin{proof}
At the chain level, the map is upper triangular with respect to the action filtration; it suffices therefore to show that the entries along the diagonal at $\pm 1$. These entries count constant holomorphic discs at the intersection points of $L_0$ with $L_1$  (together with a negative gradient flow line to the minimum).
\end{proof}
More generally, one can show that the unit $e \in HF^*(L,L;\Z)$ acts as a $2$-sided identity for the product.

The product is part of the $A_\infty$ structure on the Fukaya category, as discussed in \cite[Section 12]{MR2441780}. We shall only need the fact that it is homotopy commutative.  Concretely, the count of holomorphic discs with boundary mapping to a quadruple of Lagrangian branes $(L_0, L_1, L_2, L_3)$ defines a homotopy $\mu_3$ for the diagram
\begin{equation} \label{eq:Floer_product_associative}
  \xymatrix{ CF^*(L_2, L_3; \Z)  \otimes CF^*(L_1,L_2;\Z) \otimes CF^*(L_0, L_1; \Z) \ar[r] \ar[d] & \ar[d]  CF^*(L_2, L_3; \Z)  \otimes CF^*(L_0,L_2;\Z) \\
 CF^*(L_1,L_3;\Z) \otimes CF^*(L_0, L_1; \Z) \ar[r] &  CF^*(L_0,L_3;\Z).}
\end{equation}
This homotopy again preserves the action filtration.

\begin{definition}
A pair of Lagrangians $L_0$ and $L_1$ are \emph{Floer theoretically equivalent} if there exist cocycles $\alpha \in  CF^*(L_0,L_1;\Z)$ and $\beta \in  CF^*(L_1,L_0;\Z)$ such that 
\begin{align}
  [\beta] \cdot [\alpha] & = [e_0] \in HF^*(L_0, L_0; \Z) \\
 [\alpha] \cdot [\beta] & = [e_1] \in HF^*(L_1, L_1; \Z).
\end{align}
\end{definition}
Above, we have used the notation $\cdot$ to denote the  map  induced by $\mu_2$ on cohomology. The cycles $\alpha$ and $\beta$ are called (Floer theoretic) \emph{equivalences}.
\begin{remark} \label{rem:equivalence_unique}
  Whenever $L_0$ and $L_1$ are connected, the cohomology class of Floer theoretic equivalences are uniquely determined up to sign, because they induce isomorphisms
  \begin{equation}
    HF^*(L_0,L_1;\Z) \cong HF^*(L_0,L_0; \bZ) \cong HF^*(L_1,L_0;\Z),
  \end{equation}
and the group $ HF^*(L_0,L_0; \bZ)$ is isomorphic to $H^*(L; \bZ)$, which is of course of rank-$1$ in degree $0$. 
\end{remark}

Assuming for simplicity that $\alpha$ and $\beta$ are represented by unique intersection points $x$ and $y$, the condition that $L_0$ and $L_1$ are Floer-theoretically equivalent can be equivalently stated as the fact that the signed count of elements of $\cM(x,y)$ whose boundary passes through any fixed point of $L_0$ is $1$, as is the count of such elements passing through a fixed point of $L_1$ (see Figure~\ref{fig:pass}).
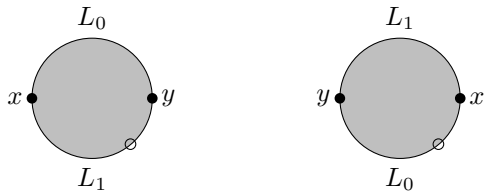
\begin{figure}[ht]
  \centering
  \begin{tikzpicture}
    \draw[fill=lightgray] (-1.8,0) circle (0.8);
    \draw (-1.8,0.8) node [above] {$L_0$};
    \draw (-1.8,-0.8) node [below] {$L_1$};
    \fill (-2.6,0) circle (2pt) node [left] {$x$};
    \fill (-1.0,0) circle (2pt) node [right] {$y$};
    \draw (-1.29,-0.61) circle (2pt);
  \end{tikzpicture}  {$\qquad\qquad$}
  \begin{tikzpicture}
    \draw[fill=lightgray] (-1.8,0) circle (0.8);
    \draw (-1.8,0.8) node [above] {$L_1$};
    \draw (-1.8,-0.8) node [below] {$L_0$};
    \fill (-2.6,0) circle (2pt) node [left] {$y$};
    \fill (-1.0,0) circle (2pt) node [right] {$x$};
    \draw (-1.29,-0.61) circle (2pt);
  \end{tikzpicture}
  \caption{The circled point at the boundary goes through a fixed point on the Lagrangian.} \label{fig:pass}
\end{figure}
In the general case, one requires instead that a linear combination of such counts equal $1$.

It follows immediately from the associativity of multiplication at the cohomological level that Floer theoretic equivalence is transitive. We record a precise formulation for future use:
\begin{lemma} \label{lem:equivalence_transitive}
If $(\alpha_1, \beta_1)$ and $(\alpha_2,\beta_2)$ are Floer theoretic equivalences for pairs $(L_0, L_1)$ and $(L_0, L_2)$, then $(\alpha_2 \cdot \alpha_1, \beta_1 \cdot \beta_2 )$ are Floer theoretic equivalences for the pair  $(L_0, L_2)$.
\end{lemma}

The most geometric way of producing Floer theoretic equivalences arises from Hamiltonian isotopies $\{ K_t \}_{t \in [0,1]}$. In this case, Floer's invariance result implies \cite{MR965228} that the Lagrangian Floer cohomology is isomorphic to ordinary cohomology, hence we have: 
\begin{align} \label{eq:6}
  HF^0(K_0,K_1;\Z) \cong \Z \cong HF^0(K_1,K_0;\Z) .
\end{align}
If the isotopy is sufficiently $C^2$-small, a minor variant of Floer's argument implies that the product
\begin{equation}
 HF^0(K_0,K_1;\Z)  \otimes  HF^0(K_1,K_0;\Z) \to HF^0(K_0,K_0;\Z) 
\end{equation}
is an isomorphism, hence that $K_0$ and $K_1$ are Floer-theoretically isomorphic. Decomposing an arbitrary path into a concatenation of short ones, and using Lemma \ref{lem:equivalence_transitive},  we conclude that any two Hamiltonian isotopic Lagrangians are Floer theoretically equivalent.
\begin{remark}
We note that Floer's original argument for invariance was rather delicate, using bifurcation analysis at birth-death singularities. This can be replaced by the most abstract method of continuation maps, or moving Lagrangian boundary condition \cite[Section 8k]{MR2441780}: one can define an cocyle in $CF^0(K_1,K_0;\Z)$ by counting holomorphic discs with boundary conditions given by the path $\{ K_t \}_{t \in [0,1]}$. A path of boundary conditions gives a homotopy between the cocyclies defined by its endpoints, so that the corresponding element of Floer cohomology is well-defined. The proof that such elements give rise to Floer theoretic isomorphisms amounts to the fact that the concatenation of a path with the inverse path is homotopic to the constant path.
\end{remark}

\subsection{Floer chains with coefficients in the group ring of $M$}
\label{sec:floer-chains-with}

Let us now, fix a base-point $\star \in M$ and consider the local coefficient system $R$ on $M$ whose value at a point $x$ is the free abelian group $R_x$ on the components of the space of paths from $\star$ to $x$. The fibre at the basepoint is the group ring $\Z[\pi_1(M, \star)]$, which acts on all other fibres by concatenation. We shall define a Floer \emph{homology group} with values in the coefficient system $R$.

Given transverse Lagrangian branes $L_0$ and $L_1$ as before, the underlying graded $R$ module for this vector space is
\begin{align} \label{eq:4}
  CF_{-i}(L_0,L_1;R) = \bigoplus_{\substack{x\in K\cap L \\ \deg(x) = i}} |\ro_x|^{\vee} \otimes R_x,
\end{align}
where $ |\ro_x|^{\vee} $ is the linear dual of the graded line $|\ro_x|$ appearing in Equation \eqref{eq:Floer_cochains}. The differential is defined as before by counting rigid elements of moduli spaces of rigid strips $\cM(x,y)$, noting that such a map $u$ induces (by concatenation with the path $u(\bR \times \{t\})$) a $\Z[\pi_1(M, \star)] $-linear map
\begin{equation*}
R_x \to R_y.
\end{equation*}
Taking the tensor product with the map induced by Equation \eqref{eq:map_induced_strip} on dual orientation lines, we obtain 
\begin{align*}
  \partial_u : |o_x|^{\vee} \otimes R_x \to |o_y|^\vee \otimes R_y.
\end{align*}

By construction, Floer chains are equipped with a filtration by subcomplexes
\begin{equation}
   F^a C_*(L_0,L_1;R) = \bigoplus_{\substack{x\in K\cap L \\ \cA(x) \leq -a}} |\ro_x|^{\vee} \otimes R_x.
\end{equation}



We extend this definition as before to the case $L_0 = L = L_1 $ by picking a Morse function $h$ on $L$, and defining
\begin{equation}
CF_*(L, L; R) = CM_{-*}(h ; R).
\end{equation}
In other words, a generator of this complex corresponds to a choice of a homotopy class of paths from a critical point of $h$ to the basepoint in $M$. The filtration on the self-Floer group is trivial, i.e. supported at the zero level.

The construction of operations on the Fukaya category naturally extends to give rise to an $A_\infty$ bimodule structure on such groups. Of key importance are the first two structure maps ($\rho$ and $\lambda$ stand for right and left multiplication)
\begin{align*}
 CF^*(L_0,L_{1};\Z)  \otimes  CF_*(L_0,L_2;R)    & \xrightarrow{\lambda} CF_*(L_1,L_2;R) \\
 CF_*(L_2,L_0;R) \otimes  CF^*(L_1,L_0;\Z)   & \xrightarrow{\rho} CF_*(L_2,L_1;R).
\end{align*}
These maps are induced by the moduli spaces of holomorphic triangles as follows: given intersection points   $x_0 \in L_0 \cap L_2$,  $x_1 \in L_0 \cap L_1$ and $ x_2 \in L_1 \cap L_2$, any element of the moduli space $\cM(x_0, x_1, x_2)$ determines a canonical homotopy class of paths from $x_0$ to $x_2$. Left and right multiplication are then defined by taking the tensor product of the induced map $R_{x_0} \to R_{x_2}$ with the map induced on orientation lines. By construction, these are maps of filtered complexes.

The right action is compatible with the product $\mu_2$ in the sense that,  given a quadruple $(L_0, L_1, L_2, L_3)$, the count of holomorphic discs with $4$ marked points defines a filtered homotopy in the following diagram:
\begin{equation} \label{eq:Floer_homology_right_module}
  \xymatrix{ CF_*(L_3, L_0; R)  \otimes  CF^*(L_1, L_0; \Z) \otimes CF^*(L_2,L_1;\Z) \ar[r] \ar[d] & \ar[d]  CF_*(L_3, L_0; R)  \otimes CF^*(L_2,L_0;\Z) \\
 CF_*(L_3,L_1;R) \otimes CF^*(L_2, L_1; \Z) \ar[r] &  CF_*(L_3,L_2;R).}
\end{equation}

The left action is similarly compatible via a diagram which we omit. Finally, we shall use the fact that the left and right actions are homotopically compatible with each other, i.e. we obtain a bimodule by passing to cohomology. Concretely, this is given by the homotopy in the diagram
\begin{equation} \label{eq:Floer_homology_bimodule}
  \xymatrix{ CF^*(L_0, L_1; \Z)  \otimes CF_*(L_0,L_3;R) \otimes CF^*(L_2, L_3; \Z) \ar[r] \ar[d] & \ar[d] CF_*(L_1,L_3;R) \otimes  CF^*(L_2, L_3; \Z)   \\
 CF^*(L_0, L_1; \Z)  \otimes CF_*(L_0,L_2;R) \ar[r] &  CF_*(L_1,L_2;R).}
\end{equation}

\subsection{The torsion of a Floer theoretic equivalence}
\label{sec:tors-floer-theor}

Consider a pair of connected Lagrangians $K$ and $Q$, and cocycles $\alpha \in CF^*(K,Q;\Z)$ and $\beta \in  CF^*(Q,K;\Z)$ which give rise to a Floer theoretic equivalence.

Using the bimodule structure, we obtain chain homotopy equivalences
\begin{align} \label{eq:1}
  \lambda(\alpha \otimes - ) : CF_*(K,L;R) & \to CF_*(Q,L;R) \\
   \lambda(\beta \otimes - ): CF_*(Q,L;R) & \to CF_*(K,L;R) \label{eq:8}
\end{align}
which are quasi-inverses, and similarly for the maps induced by $\rho(-\otimes\alpha)$ and $\rho(-\otimes\beta)$.

\begin{definition}
The \emph{Floer-theoretic torsion} of a pair $(K, L)$ of Floer-theoretically equivalent Lagrangians is the torsion of the map
\begin{equation}
CF_*(L,L; R) \to   CF_*(K,K; R)
\end{equation}
induced by left and right multiplication by equivalences.
\end{definition}

Explicitly, if $\alpha \in CF^*(L, K; \bZ)$ and $\beta \in  CF^*(K, L; \bZ)$ are equivalences, then we consider the composition
\begin{equation}
CF_*(L,L; R)\xrightarrow{\lambda(\alpha \otimes - )}   CF_*(K,L; R) \xrightarrow{\rho(- \otimes \beta)}  CF_*(K,K; R).
\end{equation}

\begin{lemma}
  \label{cor:FloerWhitehead:1}
The Whitehead torsion of $\lambda(\alpha \otimes - )$ is independent of the choice of Floer theoretic equivalence and almost complex structure. Similarly for map $\rho(-\otimes\alpha)$.
\end{lemma}
\begin{proof}
Changing $\alpha$ to $\alpha + d\alpha'$: the map $\lambda(\alpha' \otimes -)$ provides a chain homotopy between $\lambda(\alpha \otimes -)$ and $\lambda((\alpha + d\alpha') \otimes -)$, which by Lemma~\ref{lem:Whitehead:2} does not change the Whitehead torsion. By Remark \ref{rem:equivalence_unique}, a Floer theoretic equivalences is uniquely determined up to sign, which does not affect the torsion.

The continuation map for complex structures is defined using the first picture in Figure~\ref{fig:dwg}. The resulting matrix, with respect to the action filtration, is upper triangular with entries $1$ along the diagonal as in the proof of Lemma \ref{lem:continuation_chain_iso}. Combined with Lemma~\ref{lem:Whitehead:2} we see that the Whitehead torsion of the continuation map is trivial. Since composing $\lambda(\alpha \otimes -)$ with a continuation map gives a chain homotopic map, Lemma~\ref{lem:Whitehead:2} and Equation~\eqref{eq:5} prove that the torsion is independent of the choice of almost complex structure.
\end{proof}

\begin{corollary}
The Floer-theoretic torsion is independent of the choice of equivalences and almost complex structures.
\end{corollary}

\subsection{Torsion is Hamiltonian invariant}

We shall presently strengthen Lemma \ref{cor:FloerWhitehead:1} by showing  that torsion is in fact invariant under Hamiltonian isotopies. To this end, we consider Hamiltonian isotopic Lagrangians $K_0$ and $K_1$, and a fixed Lagrangian $L$  which is transverse to both. We begin by choosing a generic Hamiltonian isotopy  $\{ K_t \}_{t \in I}$ so that the following conditions hold:
  \begin{enumerate}
  \item The Lagrangians $K_t$  and $K_s$ are transverse for almost all $(s,t)\in I\times I$.
  \item There are finitely many $t\in I$ for which $K_t\cap L$ is non-transverse.
  \item For each $t\in I$ there are finitely many points in $K_t\cap L$.
  \end{enumerate} 
Consider an arbitrary member  $K_t$ in this isotopy. Choosing $\epsilon$ sufficiently small, we may assume that $K_{t \pm \epsilon} \cap L$ is arbitrarily close to $K_{t} \cap L$. In particular, since the intersection points are isolated, each intersection point of  $K_{t \pm \epsilon}$ with $ L$ is close to a unique intersection point of $K_t$ with $L$. We obtain a decomposition:
\begin{equation}
  K_{t\pm \epsilon} \cap L = \bigcup_{x \in  K_t\cap L} B^{\pm}_{x},
\end{equation}
where $ B^{\pm}_{x}$ consists of those intersection points which are close to $x$. 
\begin{lemma} \label{lem:bound_action_nearby}
For each $t \in [0,1]$ and  $\delta > 0$ which is sufficiently small, there is a constant $\epsilon$ so that 
\begin{itemize}
\item $K_{t + \epsilon}$ and $K_{t - \epsilon}$ are $\delta$-close to $K_t$ and mutually transverse, with $|\cA(z)| \ll \delta$  for all $z \in K_{t + \epsilon} \cap K_{t - \epsilon}$.
\item For any intersection point $x_{\pm} \in B^{\pm}_{x}$, we have $| \cA(x_{\pm}) - \cA(x)|  \ll  \delta$.
\end{itemize}
\end{lemma}
\begin{proof}
Given any $t \in [0,1]$,  we can choose $\epsilon$ sufficiently small so that the Lagrangians $(K_{t + \epsilon}, L, K_{t - \epsilon})$ are pairwise transverse. Since the action and the distance to an intersection point of $L$ with $K_t$ are continuous in $\epsilon$, the desired condition can be achieved for $\epsilon$ sufficiently small.
\end{proof}
The above result, together with the monotonicity lemma, gives the following constraint on moduli spaces of holomorphic triangles:
\begin{corollary} \label{cor:no_triangles_for_action}
Given distinct intersection points $x$ and $y$ of $L$ and $K_t$, with $\cA(x) \leq \cA(y)$, and $\epsilon$ sufficiently small, the moduli space $\cM(x_+,z,y_-)$ is empty for each $z \in K_{t-\epsilon} \cap  K_{t + \epsilon}$, $x_+ \in B^{+}_{x}$, and $y_- \in  B^{-}_{y} $.  Moreover, 
there is a contractible neighbourhood $U_x$ of each intersection point $x$ such that if $x_\pm \in B^{\pm}_{x}$ then any element of  $\cM(x_+,z,x_-)$ is homotopic to a map into $U_x$ relative the two marked points which map to $x_+$ and $x_-$. 
\end{corollary}
\begin{proof}
Choose disjoint neighbourhoods $U_x$ of the intersection points of $K_t$ with $L$, so that the conclusion of Lemma \ref{lem:FloerWhitehead:3} holds for $U_x$, then choose $\epsilon$ small enough so that $B^{\pm}_x \subset U_x$, and the conclusions of Lemma \ref{lem:bound_action_nearby} hold.

Whenever $x \neq y$, Stokes's theorem implies that the energy of any element of $\cM(x_+,z,y_-)$ is given by
\begin{multline}
  \cA(x_+) -  \cA(y_-)  -  \cA(z) = \\
   \left(\cA(x_+) - \cA(x)\right) + \left(\cA(x) - \cA(y)\right) + \left(\cA(y) - \cA(y_-)\right)  -  \cA(z)  \ll \delta,
\end{multline}
where the second inequality follows from the conclusion of Lemma \ref{lem:bound_action_nearby}, and the assumption that $ \cA(x) \leq \cA(y)$. On the other hand, Lemma \ref{lem:FloerWhitehead:3} implies that the energy must be greater than $\delta$, so the moduli space is empty.

The second statement follows from Corollary~\ref{cor:Monotonicity:1}.
\end{proof}



We now arrive at the main result of this section:
\begin{proposition}
  \label{lem:FloerWhitehead:1}
The  Whitehead torsion of a Floer theoretic equivalence of Hamiltonian isotopic Lagrangians is trivial.
\end{proposition}
\begin{proof}
Choosing a Hamiltonian isotopy as above, it suffices to prove that the continuation map
  \begin{align} \label{eq:7}
    CF_*(K_{t-\epsilon},L;R) \to CF_*(K_{t+\epsilon},L;R)
  \end{align}
has trivial Whitehead torsion.

Consider the action filtration given by the critical values of the action functional. If $c_0 < \cdots < c_d$ are the values of the action functional on $K_t \cap L$, we put $a(p) = -\tfrac{c_p-c_{p-1}}{2}$ (with $a(0)>-c_0$ and $a(d+1)<-c_d$) and consider the corresponding action filtration. We write
\begin{equation}
  G^pCF_*(K_{t \pm \epsilon},L;R) \equiv  F^{a(p)}CF_*(K_{t \pm \epsilon},L;R)/F^{a(p+1)}CF_*(K_{t \pm \epsilon},L;R)
\end{equation}
for the associated graded group. By Corollary \ref{cor:no_triangles_for_action}, the chain homotopy equivalence induced on the quotients
  \begin{align*}
    G^pCF_*(K_{t-\epsilon},L;R)  \to G^pCF_*(K_{t+\epsilon},L;R)
  \end{align*}
  splits as a direct sum - one for each of the original intersection points in $K_t\cap L$. Moreover, there is a fixed contractible set which contains the image (up to homotopy relative the essential marked points) of each holomorphic triangle which contributes to the map of associated graded spaces.
This means that the continuation map on these quotients can be identified with the standard continuation map
  \begin{align*}
    G^pCF_*(K_{t-\epsilon},L;\Z) \to G^pCF_*(K_{t+\epsilon},L;\Z)
  \end{align*}
  tensored with $R$ (and multiplication with the identity) - hence it has trivial Whitehead torsion. So, Lemma~\ref{lem:Whitehead:1} shows that the entire continuation map in Equation~\eqref{eq:7} has trivial torsion.
\end{proof}

The above result has the following straightforward generalisation:
\begin{corollary}
  \label{cor:FloerWhitehead:1}
The Whitehead torsions of Equation \eqref{eq:1} and \eqref{eq:8} are invariant under Hamiltonian isotopy of $K$, $L$ and $Q$.
\end{corollary}
\begin{proof}
  Let $L_0$ and $L_1$ be Hamiltonian isotopic. As a special case of Equation \eqref{eq:Floer_homology_bimodule}, we have a chain homotopy commutative diagram 
  \begin{align*}
    \xymatrix{
      CF_*(K,L_0;R) \ar[r]^{\lambda(\alpha\otimes-)} \ar[d]^{\rho(-\otimes \beta)} & CF_*(Q,L_0;R) \ar[d]^{\rho(-\otimes \beta)} \\
      CF_*(K,L_1;R) \ar[r]^{\lambda(\alpha\otimes-)} & CF_*(Q,L_1;R) \\      
    }
  \end{align*}
  where the vertical maps are the continuation maps considered in Lemma~\ref{lem:FloerWhitehead:1} which up to homotopy do not depend on the choice of $\beta \in CF^0(L_0,L_1;\Z)$. Using Lemma~\ref{lem:FloerWhitehead:1}, Lemma~\ref{lem:Whitehead:2} and Equation~\eqref{eq:5} we see that the two horizontal maps have the same Whitehead torsion.

  Changing $K$ or $Q$ involves slightly different diagrams:
  \begin{align*}
    \xymatrix{
      CF_*(K,L;R) \ar[r] \ar[rd] & CF_*(Q_0,L;R) \ar[d]\\
      & CF_*(Q_1,L;R) \\
    } \qquad 
    \xymatrix{
      CF_*(K_0,L;R) \ar[r] \ar[d] & CF_*(Q,L;R)\\
      CF_*(K_1,L;R) \ar[ur] \\
    } 
  \end{align*}
  but the conclusion is the same since in each case the additional map has trivial torsion by Lemma~\ref{lem:FloerWhitehead:1}, and the diagrams homotopy commute by using Equation \eqref{eq:Floer_homology_right_module} and its analogue for the left module action.
\end{proof}

\section{The Floer-theoretic torsion vanishes} \label{sec:spectr-sequ}

In this section we specialise to the setting of cotangent bundles: recall that, if $N$ is a closed manifold, every closed exact Lagrangian in $T^*N$ is Floer-theoretically equivalent to the $0$-section (\cite{Nadler,FSS,Abou1,MySympfib}). Our main result will be the following:
\begin{proposition}
If $L \subset T^*N$ is an exact Lagrangian, then the torsion of the Floer-theoretic equivalence with the $0$-section vanishes. 
\end{proposition}

We shall prove this result by appealing to the Hamiltonian invariance of torsion and a large-scale deformation of the Lagrangian $L$. The deformation of $L$ is essentially the same as that used in \cite{HomotopySerre}. 

Denote by $\psi^r$ the time $\log r$ Liouville flow (which corresponds to dilating the fibres by $r$), and by  $\phi_g$ the Hamiltonian flow of $g \circ q$, where  $g  \co N \to \R$ is a Morse function with distinct critical values. We define
\begin{align} \label{eq:2}\begin{array}{l@{}l}
    L^r &= \psi^r L \\
    L^r_1 &= \phi_{g} L^r,
\end{array}
\end{align}
and note that Corollary \ref{cor:FloerWhitehead:1} implies that the Floer theoretic torsion of the equivalence of $L$ with the zero section agrees with the torsion of the map
\begin{equation} \label{eq:deform_L_N_by_large_and_small}
 CF_*(L^r,L^r_1; R) \to   CF_*(N,N_1; R),
\end{equation}
induced by left and right multiplication with Floer theoretic equivalences $\alpha^r \in  CF^*(L^r,N)$ and $\beta^r_1 \in  CF^*(L^r_1,N_1)$, noting the fact that $N$ is invariant under the Liouville flow, hence $N^r = N$.

By construction, intersection points of $N$ and $N_1$ occur at critical points of $g$, and $\cA(x) = - g(x)$. In particular, if $-c_d < \cdots < -c_0$ are the critical values of $g$ then we consider the action filtration
\begin{equation}
  F^{a(p)}  CF_*(N,N_1; R)
\end{equation}
using the actions $a(p)=-\frac{c_p - c_{p-1}}{2} $, the fact that there is a unique intersection $x_p$ of action $c_p $ implies that the associated graded group is
\begin{equation}
    G^p  CF_*(N,N_1; R) \cong R_{x_p}.
\end{equation}

We shall now assume that $r \ll 1$ in which case we have:
\begin{itemize}
\item The cocycles $\alpha^r$  and $\beta^r_1$ are supported in action level $\ll 1$.
\item The intersection points of $L^r $ and  $L^r_1 $ decompose as a disjoint union
  \begin{equation}
 L^r \cap L^r_1 =  \coprod_{p=0}^{d} B_{p}^{r}, 
  \end{equation}
with $B_p^{r}$ consisting of those intersection points lying in a small neighbourhood of the critical point $x_p$ of $g$.
\item The action of any element $y \in  B_{p}^{r}$ satisfies
  \begin{equation} \label{eq:action_almost_1cp}
  |\cA(y) - c_p| \ll 1.   
  \end{equation}
\end{itemize}

We now consider the filtration on $  CF_*(L^r,L^r_1; R) $ given by the same action levels as the above filtration on $ CF_*(N,N_1; R)  $, i.e. $\frac{c_p - c_{p-1}}{2} $. The $p$\th associated graded group is generated by intersection points of action between $ \frac{c_p - c_{p-1}}{2}  $ and $\frac{c_{p+1} - c_{p}}{2} $. By Equation \eqref{eq:action_almost_1cp}, all such intersection points in fact have action approximately equal to $c_p$ and lie in the set $ B_{p}^{r}$ of intersection points which are close to $x_p$. This associated graded group is therefore given by
\begin{equation}
    G^p  CF_*(L^r,L^r_1; R) \cong  \bigoplus_{y \in  B_p^{r}} R_{y},
\end{equation}
and is equipped with the induced differential.

Since left and right multiplication induce maps which preserve action filtrations, and the equivalences  $\alpha^r$  and $\beta^r_1$ have small action, we conclude that Equation \eqref{eq:deform_L_N_by_large_and_small} descends to the associated graded with respect to the filtration, as does the chain homotopy inverse induced by multiplication with the quasi-inverses of $\alpha^r$  and $\beta^r_1$. We conclude:
\begin{lemma}
The equivalences  $\alpha^r$  and $\beta^r_1$ induce a chain homotopy equivalence
 \begin{equation} \label{eq:map_associated_graded_large_defor}
 G^p  CF_*(L^r,L^r_1; R)  \to   G^p  CF_*(N,N_1; R).
\end{equation} 
\qed
\end{lemma}

On the other hand, Equation \eqref{eq:map_associated_graded_large_defor} agrees with the map obtained from
 \begin{equation} \label{eq:10} 
 G^p  CF_*(L^r,L^r_1; \Z)  \to   G^p  CF_*(N,N_1; \Z)
\end{equation}
by extension of scalars from $\Z$ to $R$. Indeed, using Corollary \ref{cor:Monotonicity:1} with $K=N$ and $Q=N_1$, we see (as in the proof of Proposition \ref{lem:FloerWhitehead:1}) that the homotopy class of the path defined by any rigid disc counted in these maps is homotopic relative to end points to a path inside a small contractible neighbourhood of $x_p$.

In particular, the Whitehead torsion of Equation \eqref{eq:map_associated_graded_large_defor} vanishes. Since this is the map on associated graded groups of Equation \eqref{eq:deform_L_N_by_large_and_small}, we conclude from Lemma \ref{lem:Whitehead:1} that the Whitehead torsion vanishes. Using invariance of torsion under Hamiltonian isotopies, we arrive at the main result of this section:
\begin{proposition} \label{prop:floer_torsion_vanishes}
The Whitehead torsion of any Floer-theoretic equivalence of closed exact Lagrangians in $T^*N$ vanishes. \qed
\end{proposition}

\section{From cellular to Floer cochains} \label{sec:from-cellular-floer}

In this section, we prove that the Whitehead torsion of the projection map $ L \to N $, which we call the \emph{classical torsion},  agrees with the Floer theoretic torsion.
We will prove this by passing through a sequence of intermediary maps which are defined using Morse and Floer theory in disc cotangent bundle $D^*N \subset T^*N$ (which we assume contains $L$ in its interior).


On the Morse theoretic side,  we consider the following specific context:  let $h:L \to \R$ be any Morse function on $L$, which we extend to a small tubular neighborhood of $L$ inside $D^*N$ by 
\begin{equation} \label{eq:function_near_L}
  H(z)=h(\pi(z))+d_L(z)^2 
\end{equation}
where $d_L(-)$ denotes the distance to $L$ with respect to some Riemannian metric on $D^*N$, and $\pi(z)$ is the nearest point to $z$ in $L$. Each critical point of $h$ gives rise to a critical point of $H$ whose descending manifolds agree (i.e. the descending manifolds with respect to $H$ are contained in $L$). Extend this to a Morse function $H:D^*N \to \R$ which has outwards pointing gradient at the boundary, we conclude that the Morse complex $CM_*(L,H;R)$ embeds as a sub-complex into $CM_*(D^*N,H;R)$.



\begin{lemma}\label{lem:CWtoFloer:1}
The classical torsion agrees with the torsion of the inclusion
\begin{equation} \label{eq:map_Morse_complex}
CM_{*}(L;h,R) \to CM_*(D^* N;H,R).
\end{equation}
\end{lemma}

\begin{proof}
We first use the factorisation
  \begin{equation}
   L \to D^* N \to  N,
  \end{equation}
to see that the classical torsion agrees with the torsion of the map $L \to D^*N$ we note that the second map has trivial torsion since $D^*N $ has a sequence of simple collapses to $N$.

We now appeal to the fact that union of the descending manifolds of $H$ yield a CW subcomplex $X_{D^*N}  \to D^* N$, so that we obtain a further factorisation
\begin{equation}
 L \to X_{D^* N} \to D^* N.
\end{equation}
We claim that the inclusion $X_{D^* N} \to D^* N$ has trivial torsion. Indeed, this is in fact a sequence of trivial expansions, as can be seen by deforming the Morse function $H$ outside a compact set so that its restriction to a neighbourhood of $\partial D^*N$ agrees with
\begin{equation}
  f(\pi(z)) - d_{ \partial D^*N}(z),
\end{equation}
with $f$ a Morse function on $ \partial D^*N$. The descending manifolds give rise to a cell-decomposition of $D^*N$ as $X_{D^*N}$ with two extra cells attached for each critical point of $f$. This collapses to $X_{D^* N}$ by starting with the top dimensional cells and going down in dimension. Combining the above arguments, we conclude that the torsion of $L \to N$ agrees with the torsion of $L \to X_{D^* N}$.

To complete the argument, we observe that the map $ L \to X_{D^* N} $ is cellular with respect to the Morse-theoretic CW structure on $L$, and we have a natural commutative diagram
\begin{equation}
     \xymatrix{
     C_*^{CW}(L;R) \ar[r] \ar[d]^{\cong} & C_*^{CW}(X_{D^*N};R) \ar[d]^{\cong} \\
    CM_*(L,h;R) \ar[r] & CM_*(D^*N,H;R).
    }
\end{equation}

\end{proof}

We have arranged our fixed choice of pair of Morse functions so that Equation \eqref{eq:map_Morse_complex} is an inclusion. This inclusion can also be interpreted as a count of configurations shown in Figure~\ref{fig:maps:2}, which consist of: (1) a flow line $\gamma_-$ of $h$ parametrised by $(-\infty,0]$ (2) a flow line $\gamma_+$ of $H$ parametrised by $[0,\infty)$, and (3) holomorphic map $u \co D^2 \to M$ mapping the boundary to $L$, such that $u(-1) = \gamma_-(0)$ and $u(0) = \gamma_+(0)$. The key point is that all such rigid configurations are constant: the exactness of $L$ implies that all holomorphic discs are constant, and Equation \eqref{eq:function_near_L} implies that all negative gradient flow lines of $H$ starting on $L$ are contained in $L$.
\begin{figure}[ht]
  \centering
  \begin{tikzpicture}
    \draw[fill=lightgray] (1.8,0) circle (0.8);
    \draw[->] (0,0) -- (0.5,0) node [above] {$h$} node [below] {$L$};
    \draw (0.5,0) -- (1,0);
    \draw (2.3,-0.4) node [below right] {$L$};
    \fill (0,0) circle (2pt);
    \draw[dotted,->,very thick] (1.8,0) --  (2.6,0.8) node [above] {$H$};
    \fill (2.6,0.8) circle (2pt);
    \draw[dashed] (-0.3,-1.4) -- (-0.3,1.4) -- (2.95,1.4) -- (2.95,-1.4) -- (-0.3,-1.4); 
  \end{tikzpicture}
  \caption{Standard Morse map in Lemma~\ref{lem:CWtoFloer:1} interpreted as a count of discs and gradient flow lines: the dotted line is a flow lines of  $H$ starting at the center of the disc and converging to a critical point.} \label{fig:maps:2}
\end{figure}
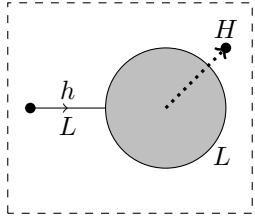

\begin{proposition} \label{prop:CWtoFloer:1}
    The classical and Floer-theoretic torsions agree.
\end{proposition}

\begin{proof} 
Consider the composition
\begin{equation}\label{eq:9}
  CM_*(L; h,R) \xrightarrow{\rho} CF_*(L,N) \xrightarrow{\lambda} CM_*(N,N;R) \xrightarrow{m} CM_*(D^*N; H, R)
\end{equation}
By Lemma~\ref{lem:CWtoFloer:1} the last map has trivial Whitehead torsion. Indeed, the standard inclusion $N\subset D^*N$ has trivial Whitehead torsion. The proposition will thus follow from the claim that the composition of these is homotopic to the map from Lemma~\ref{lem:CWtoFloer:1}, which is represented in Figure~\ref{fig:maps:2}.

Let $\alpha\in CF^*(L,N)$ and $\beta \in CF^*(N,L)$ be Floer cochains inducing the Floer theoretic equivalences, and let $f:N \to \R$ be any Morse function on $N$ which has generic intersection of its stable manifolds with the unstable manifolds of $H$. The maps in the sequence are given by counting rigid objects in moduli spaces of maps as in Figure~\ref{fig:maps:1}.

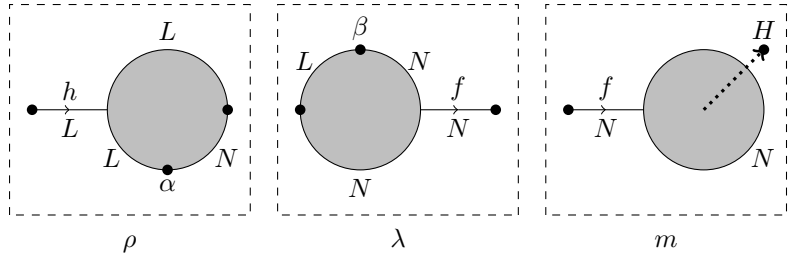
\begin{figure}[ht]
  \centering
  \begin{tabular}{ccc}
  \begin{tikzpicture}
    \draw[fill=lightgray] (1.8,0) circle (0.8);
    \draw[->] (0,0) -- (0.5,0) node [above] {$h$} node [below] {$L$};
    \draw (0.5,0) -- (1,0);
    \draw (1.8,0.8) node [above] {$L$};
    \draw (1.3,-0.4) node [below left] {$L$};
    \draw (2.3,-0.4) node [below right] {$N$};
    \fill (0,0) circle (2pt);
    \fill (2.6,0) circle (2pt);
    \fill (1.8,-0.8) circle (2pt) node [below] {$\alpha$};
    \draw[dashed] (-0.3,-1.4) -- (-0.3,1.4) -- (2.9,1.4) -- (2.9,-1.4) -- (-0.3,-1.4); 
  \end{tikzpicture} & 
  \begin{tikzpicture}
    \draw[fill=lightgray] (-1.8,0) circle (0.8);
    \draw (0,0) -- (-0.5,0);
    \draw[<-] (-0.5,0) node [above] {$f$} node [below] {$N$} -- (-1,0);
    \fill (-0,0) circle (2pt);
    \draw (-1.8,-0.8) node [below] {$N$};
    \draw (-1.3,0.4) node [above right] {$N$};
    \draw (-2.3,0.4) node [above left] {$L$};
    \fill (-2.6,0) circle (2pt);
    \fill (-1.8,0.8) circle (2pt) node [above] {$\beta$};
    \draw[dashed] (0.3,-1.4) -- (0.3,1.4) -- (-2.9,1.4) -- (-2.9,-1.4) -- (0.3,-1.4); 
  \end{tikzpicture} &
  \begin{tikzpicture}
    \draw[fill=lightgray] (1.8,0) circle (0.8);
    \draw[->] (0,0) -- (0.5,0) node [above] {$f$} node [below] {$N$};
    \draw (0.5,0) -- (1,0);
    \draw (2.3,-0.4) node [below right] {$N$};
    \fill (0,0) circle (2pt);
    \draw[dotted,->,very thick] (1.8,0) --  (2.6,0.8) node [above] {$H$};
    \draw[dashed] (-0.3,-1.4) -- (-0.3,1.4) -- (2.9,1.4) -- (2.9,-1.4) -- (-0.3,-1.4); 
    \fill (2.6,0.8) circle (2pt);
  \end{tikzpicture} \\
    $\rho$ & $\lambda$ & $m$
  \end{tabular}  
  \caption{The three maps in Equation~\eqref{eq:9}} \label{fig:maps:1}
\end{figure}
Here marked points on the boundary of a disc should be mapped to an intersection point.

To create a homotopy between the two maps we construct a $1$-parameter family of moduli spaces, parametrised by $s \in [-1,1]$, interpolating between Figures \ref{fig:maps:1} and \ref{fig:maps:2}.

\begin{itemize}
\item Near $ s= -1$, we glue the flow lines in the right and center of  Figure \ref{fig:maps:1} to obtain a (finite of length $a(s)$) flow line of $f$ on $N$. The resulting configuration of discs and flow lines is illustrated  in Figure~\ref{fig:maps:3}.
  \begin{figure}[ht]
    \centering
    \begin{tabular}{ccc}
      \begin{tikzpicture}
        \draw[fill=lightgray] (1.8,0) ellipse (0.8 and 0.8);
        \draw[fill=lightgray] (3.4,0) ellipse (0.8 and 0.8);
        \draw[->] (0,0) -- (0.5,0) node [above] {$h$} node [below] {$L$};
        \draw (0.5,0) -- (1,0);
        \draw (1.8,0.8) node [above] {$L$};
        \draw (1.3,-0.4) node [below left] {$L$};
        \draw (2.1,-0.6) node [below right] {$N$};
        \draw (3.1,0.6) node [above left] {$L$};
        \fill (1.8,-0.8) circle (2pt) node [below] {$\alpha$};
        \fill (3.4,0.8) circle (2pt) node [above] {$\beta$};
        \draw (3.4,-0.8) node [below] {$N$};
        \fill (0,0) circle (2pt);
        \fill (2.6,0) circle (2pt);
        \draw[->] (4.2,0) -- (4.7,0) node [above] {$N,f$} node [below] {$a(s)$};
        \draw (4.7,0) -- (5.2,0);
        \draw[dashed] (-0.3,-1.4) -- (-0.3,1.4) -- (7.2,1.4) -- (7.2,-1.4) -- (-0.3,-1.4); 
        \draw[fill=lightgray] (6,0) circle (0.8);
        \draw (6.5,-0.4) node [below right] {$N$};
        \fill (5.2,0) circle (2pt);
        \draw[dotted,->,very thick] (6,0) --  (6.8,0.8) node [above] {$H$};
        \fill (6.8,0.8) circle (2pt);
      \end{tikzpicture} 
    \end{tabular}
    \caption{The moduli space for $-1<s<0$} \label{fig:maps:3}
  \end{figure}
\item As $s$ approaches $0$ from the left, we let the length of the finite gradient flow line along $N$ go to zero. In particular, at $s=0$, we obtain the moduli space shown in Figure~\ref{fig:maps:threediscs}
  \begin{figure}[ht]
    \centering
    \begin{tabular}{ccc}
      \begin{tikzpicture}
        \draw[fill=lightgray] (1.8,0) ellipse (0.8 and 0.8);
        \draw[fill=lightgray] (3.4,0) ellipse (0.8 and 0.8);
        \draw[->] (0,0) -- (0.5,0) node [above] {$h$} node [below] {$L$};
        \draw (0.5,0) -- (1,0);
        \draw (1.8,0.8) node [above] {$L$};
        \draw (1.3,-0.4) node [below left] {$L$};
        \draw (2.1,-0.6) node [below right] {$N$};
        \draw (3.1,0.6) node [above left] {$L$};
        \fill (1.8,-0.8) circle (2pt) node [below] {$\alpha$};
        \fill (3.4,0.8) circle (2pt) node [above] {$\beta$};
        \draw (3.4,-0.8) node [below] {$N$};
        \fill (0,0) circle (2pt);
        \fill (2.6,0) circle (2pt);
        \draw[dashed] (-0.3,-1.4) -- (-0.3,1.4) -- (6.2,1.4) -- (6.2,-1.4) -- (-0.3,-1.4); 
        \draw[fill=lightgray] (5,0) circle (0.8);
        \draw (5.5,-0.4) node [below right] {$N$};
        \fill (4.2,0) circle (2pt);
        \draw[dotted,->,very thick] (5,0) --  (5.8,0.8) node [above] {$H$};
        \fill (5.8,0.8) circle (2pt);
      \end{tikzpicture} 
    \end{tabular}
    \caption{The moduli space for $-1<s<0$} \label{fig:maps:threediscs}
  \end{figure}
which has three disc components meeting at two boundary nodes. These three disc components represent a point in (the Gromov boundary of) the moduli space of discs with $3$ boundary marked points and one interior marked point. This moduli space is connected and has dimension $2$.
\item  For $0 < s< 1$ we pick a path in the moduli space of such discs, as shown in Figure~\ref{fig:maps:4b},
  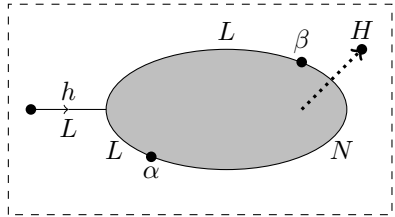
\begin{figure}[ht]
    \centering
    \begin{tabular}{ccc}
      \begin{tikzpicture}
        \draw[fill=lightgray] (2.6,0) ellipse (1.6 and 0.8);
        \draw[->] (0,0) -- (0.5,0) node [above] {$h$} node [below] {$L$};
        \draw (0.5,0) -- (1,0);
        \draw (2.6,0.8) node [above] {$L$};
        \draw (1.35,-0.3) node [below left] {$L$};
        \fill (1.6,-0.63) circle (2pt) node [below] {$\alpha$};
        \fill (3.6,0.63) circle (2pt) node [above] {$\beta$};
        \draw (3.85,-0.3) node [below right] {$N$};
        \fill (0,0) circle (2pt);
        \draw[dashed] (-0.3,-1.4) -- (-0.3,1.4) -- (4.8,1.4) -- (4.8,-1.4) -- (-0.3,-1.4); 
        \draw[dotted,->,very thick] (3.6,0) --  (4.4,0.8) node [above] {$H$};
        \fill (4.4,0.8) circle (2pt);
      \end{tikzpicture} 
    \end{tabular}
    \caption{The moduli space at $0<s<1$} \label{fig:maps:4b}
  \end{figure}
between the configuration appearing at $s=0$, and the configuration shown in Figure~\ref{fig:maps:4c}
  \begin{figure}[ht]
    \centering 
    \begin{tabular}{ccc}
      \begin{tikzpicture}
        \draw[fill=lightgray] (2.6,0) ellipse (1.6 and 0.8);
        \draw[->] (0,0) -- (0.5,0) node [above] {$h$} node [below] {$L$};
        \draw (0.5,0) -- (1,0);
        \draw (2.6,0.8) node [above] {$L$};
        \draw (2.6,-0.8) node [below] {$L$};
        \fill (0,0) circle (2pt);
        \draw[dashed] (-0.3,-1.4) -- (-0.3,1.5) -- (6.5,1.5) -- (6.5,-1.4) -- (-0.3,-1.4); 
        \draw[fill=lightgray] (5,0) circle (0.8);
        \draw (5.8,0) node [right] {$N$};
        \fill (4.2,0) circle (2pt);
        \draw[dotted,->,very thick] (2.6,0) --  (3.4,1.0) node [above] {$H$};
        \fill (3.4,1.0) circle (2pt);
        \fill (5.0,-0.8) circle (2pt) node [below] {$\alpha$};
        \fill (5.0,0.8) circle (2pt) node [above] {$\beta$};
        \draw (4.6,0.5) node [above left] {$L$};
        \draw (4.6,-0.5) node [below left] {$L$};
      \end{tikzpicture} 
    \end{tabular}
    \caption{The moduli space for $s=1$}\label{fig:maps:4c}
  \end{figure}
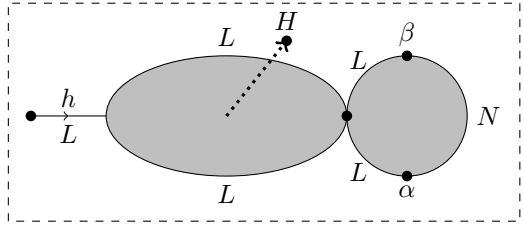
consisting of a stable disc with two components meeting at a node, one carrying an interior marked point and one boundary marked, and the other carrying two boundary marked point labelled by the equivalences $\alpha$ and $\beta$.
\end{itemize}
In this way, we have produced a homotopy between the composition appearing in Equation \eqref{eq:9}, and the map defined by the moduli spaces for $s=1$. Since $\alpha$ and $\beta$ are Floer theoretic equivalences, the (signed) count of holomorphic bigons with corners labelled $\alpha$ and $\beta$ passing thought a generic point of $L$ is $1$. This implies that the (signed) count of configurations shown in  Figure  \ref{fig:maps:4c} is the same as the count of the moduli space appearing in  Figure~\ref{fig:maps:2}, which, by Lemma \ref{lem:CWtoFloer:1}, agrees with the classical torsion.  Note in particular that each contribution to the local coefficient system is the same since this is given by the path (unique up to homotopy) from the far left of the diagram to the marked point at the end of the arrow decorated with $H$.
\end{proof}

\appendix
\section{The monotonicity Lemma} 
Let $M$ be any symplectic manifold (possibly non-compact) with a given almost complex structure $J$, and let $K$ and $Q$ be two proper Lagrangians in $M$ with an isolated intersection point $x\in K\cap Q$. 

Let $U_x$ be a neighbourhood of $x$, and consider $J$-holomorphic discs $u \co D^2 \to M$ which satisfy the following conditions with respect to a constant $0 <\delta$:
\begin{align} \label{eq:image_marked_points}
& \parbox{30em}{$u(1) \in U_x$ and $u(-1) \notin U_x$} \\ \label{eq:image_segments}
& \parbox{30em}{$u$ maps the upper semi-circle in $\partial D^2$ to a $\delta$ neighbourhood of $K$, and the lower semi-circle to a $\delta$ neighbourhood of $Q$.}
\end{align}
The following technical result about minimal areas of holomorphic curves is a variant of the monotonicity Lemma,  and is crucially used in Sections \ref{sec:whit-tors-fukaya} and \ref{sec:spectr-sequ}:
\begin{lemma} \label{lem:FloerWhitehead:3}
For each sufficiently small neighborhood $U_x$, there is a constant $\delta$ such that all holomorphic discs satisfying the Conditions \eqref{eq:image_marked_points} and \eqref{eq:image_segments} have area greater than $\delta$.
\end{lemma}
\begin{proof}
  Pick a Darboux chart $D^{2n}_\epsilon\to M$ mapping the origin to $x$ such that $Q$ is identified with $D^n_\epsilon$, and such that the image of the tangent space of $K$ at $x$ has trivial intersection with $i\R^n$ (i.e. is non-vertical). Locally $K$ is given by the differential of a function, and by making $\epsilon$ smaller we may assume that $f:D^n_\epsilon \to \R$ is such a function, and that the differential defines $K$ in a neighborhood of the type $D_\epsilon^n \times iD_{\epsilon'}^n$.

Since $f$  has an isolated critical point at $0$, we have a co-dimension 2 manifold
  \begin{align*}
    W = \{(q,p)\in D^n_\epsilon \times iD^n_{\epsilon'} \mid \norm{q}=\epsilon/2, \norm{p} = \norm{d_qf}/2\}
  \end{align*}
which is an $S^{n-1}$ bundle over the sphere of radius $\epsilon/2$ in $\bR^n$.

We shall presently appeal to the monotonicity Lemma \cite[Proposition 4.4.1]{MR1274929}: there is a constant $C$ so that any holomorphic curve which intersects $W$ and has boundary in the complement of the $C \cdot A$ neighbourhood of $W$ has area greater than $A$, whenever $A$ is sufficiently small.

Let us therefore choose some neighbourhood $U$ of $x$ which is disjoint from $W$, and pick $\delta$ so that the following properties hold in the Darboux chart:
\begin{enumerate}
\item The $\delta$ neighbourhoods of $Q$ and $K$ are disjoint away from $U$.
\item The $C \delta$ neighbourhood of $W$ is disjoint from $U$ and from the $\delta$ neighbourhoods of $Q$ and $K$.
\end{enumerate}
The first assumption implies that $u(-1)$ lies outside of the chart. The image of the boundary of $u$ under the collapse map $ M \to D^{2n}_\epsilon/ \partial D^{2n}_\epsilon$ is therefore a loop which winds once around $W$. The algebraic intersection number of $u$ with $W$ is hence non-trivial modulo 2. The image of $u$ therefore intersects $W$, while the image of its boundary lies outside the $C \delta$-neighbourhood of $W$. The monotonicity Lemma implies that such a curve has area greater than $\delta$, completing the proof.
\end{proof}

We shall also consider the following variant:

\begin{corollary} \label{cor:Monotonicity:1}
For each sufficiently small neighbourhood of $x$, there is a constant $\delta$ such that, if Condition \eqref{eq:image_segments} holds, the image of $\pm 1$ under $u$ lie in $ U_x$, and the area of $u$ is less than $\delta$, then the path $\{u(t)\}_{t \in [-1,1]}$ is homotopic relative its end points to a map with image in $U_x$.
\end{corollary}
\begin{proof}
  Assume by contradiction that there is such a disc for which the path from $u(-1)$ to $u(1)$ is not homotopic to a path with image in $U_x$, we then see that the lift of $u$ to the universal cover of $M$ has endpoints mapping to distinct inverse images of $U_x$. Replacing $K$ and $Q$ by their inverse images, we find a contradiction to the previous Lemma.
\end{proof}

\bibliographystyle{plain}
\bibliography{../Mybib}

\end{document}